\documentclass[11pt]{csdm}

\usepackage{url,floatflt,helvet,times}
\usepackage{psfig,graphics}
\usepackage{mathptmx,amssymb,amsmath,bm}
\usepackage{float,tcolorbox,xcolor}
\usepackage[bf,hypcap]{caption}
\usepackage{graphicx,wrapfig,subfig}
\usepackage{setspace,enumerate,bookmark,placeins}
\usepackage{multirow,multicol}
\usepackage[curve,frame,line,arrow,matrix]{xy}
\usepackage[latin5]{inputenc}
\usepackage{tikz,calc}
\usetikzlibrary{arrows,backgrounds,automata}
\usetikzlibrary{decorations.markings}
\tikzstyle{vertex}=[circle, draw, inner sep=2pt, minimum size=6pt]
\newcommand{\vertex}{\node[vertex]}

\emergencystretch=\maxdimen
\allowdisplaybreaks

\def\noi{\noindent}

\newcommand{\C}{\mathcal{C}}
\newcommand{\cn}{cn_{\chi}}
\newcommand{\cnc}{cn^c_{\chi}}

\def\firstpage{21}
\def\thevol{1}
\def\thenumber{1}
\def\theyear{2017}

\begin{document}
{\fontfamily{cmss}\selectfont

\titlefigurecaption{\hspace{0.3cm}{\hspace{0.6cm}\LARGE \bf \sc \sffamily\color{white} Contemporary Studies in Discrete Mathematics}}

\title{\sc \sffamily On Chromatic Curling Number of Graphs}

\author{\sc\sffamily C. Susanth$^{1,\ast}$, N.K. Sudev$^2$ and S.J. Kalayathankal$^3$}
\institute {$^{1}$Department of Mathematics, Research \& Development Centre, Bharathiar University, Coimbatore-641046, India. \\
$^{2}$Centre for Studies in Discrete Mathematics, Vidya Academy of Science \& Technology, Thrissur - 680501, Kerala, India.\\ 
$^{3}$Department of Mathematics, Kuriakose Elias College, Kottayam-686561, Kerala, India.
}


\titlerunning{On chromatic curling number of graphs}
\authorrunning{C. Susanth, N.K. Sudev \& S.J. Kalayathankal }

\mail{susanth\_c@yahoo.com}

\received{20 May 2017}
\revised{24 July 2017}
\accepted{27 August 2017}
\published{20 October 2017.}

\abstracttext{ The curling number of a graph $G$ is defined as the number of times an element in the degree sequence of $G$ appears the maximum number of times. Graph colouring is an assignment of colours, labels or weights to the vertices or edges of a graph. A colouring $\C$ of colours $c_1,c_2,\ldots,c_l$ is said to be a minimum parameter colouring if $\C$ consists of minimum number of colours with smallest subscripts. In this paper, we study colouring version of curling number of certain graphs, with respect to their minimum parameter colourings.}
\keywords{Graph colouring; curling number; compound curling number; chromatic curling number; chromatic compound curling number.}

\msc{05C15, 40B05.}
\maketitle

\section{Introduction}

For all  terms and definitions, not defined specifically in this paper, we refer to \cite{BM1,FH,DBW} and for the terminology of graph colouring, we refer to \cite{CZ1,JT1,MK1}.  Unless mentioned otherwise, all graphs considered in this paper are simple, finite, connected and undirected.

The \textit{curling number} of a graph $G$ may be defined as the number of times an element in the degree sequence of $G$ appears the most (see \cite{KSC}). That is, the curling number of a non-empty finite graph $G$, whose degree sequence $S$ is written as a string of subsequences $S=X_1^{k_1}\circ X_2^{k_2}\circ \ldots\circ X_l^{k_l}$, is defined as $cn(G)=\max\{k_i:1\le i\le l\}$, where $l$ the number of distinct elements in $S$. 

The \textit{compound curling number} of $G$ $cn^c(G)$ is defined as $cn^c(G)=\prod_{i=1}^{l}k_i$, where $1\le i\le l$ (see \cite{KSC}). The curling number and compound curling number of different graph classes, graph operations, graph powers and graph products have been studied in detail (see \cite{KSC,NKS1,SSSCK1,SSSCK2,SSSCK3}). Analogous to the notion curling number of graphs, the notion of chromatic curling number of graphs has been introduced in \cite{KSC}.

\textit{Graph colouring} is an assignment of colours, labels or weights to the vertices, edges and faces of a graph under consideration. Different graph colouring problems originated from various practical and real life situations. Unless stated otherwise, a graph colouring mentioned in this paper is an assignment of colours to the vertices of a graph subject to certain conditions. A colouring of a graph $G$, with respect to which no two adjacent vertices of $G$ have the same colour, is called a \textit{proper colouring} of $G$. 

The \textit{chromatic number} of a graph $G$, denoted by $\chi(G)$, is the minimum number of colours required in a proper colouring of the given graph. The \textit{chromatic colouring} of a graph $G$ is colouring of $G$ consisting of exactly $\chi(G)$ colours. A colouring $\C$ of colours $c_1,c_2,\ldots,c_l$ is said to be a \textit{minimum parameter colouring} if $\C$ consists of minimum number of colours with smallest subscripts (see \cite{KSM}). The set of vertices in $G$ having the colour $c_i$ is called a \textit{colour class} of $G$ and is denoted by $\C_i$. The \textit{weight} of a colour $c_i$ is the cardinality of its colour class. 

A colouring $\C$ is said to be a \textit{chromatic colouring} if $|\C|=\chi(G)$. A chromatic colouring $\C=\{c_1,c_2,\ldots,c_l\}$ is said to be a \textit{$\chi^-$-colouring} of $G$ if maximum possible number of vertices get the colour $c_1$ in $G$, maximum number of possible vertices in $G-\C_1$ get the colour $c_2$, maximum number of possible vertices in $G-\C_1\cup \C_2$ get the colour $c_3$, maximum number of possible vertices in $G-(\C_1\cup \C_2\cup \C_3)$ get the colour $c_4$ and so on.   A \textit{$\chi^+$-colouring} of $G$ is a chromatic colouring $\C$ obtained by replacing the colour $c_i$ with the colour $c_j$ in its $\chi^-$-colouring, where $i+j=\chi(G)+1$.

Motivated by the investigations mentioned above on different types graph colourings and related colouring parameters, in this paper, we investigate the chromatic version of the curling number of graphs with respect to  colouring of graphs.

\section{New results}

The chromatic curling number of a graph $G$ with respect to a chromatic colouring is defined as follows.

\begin{definition} 
Let $\C=\{c_1,c_2,\ldots, c_\ell\}$ be a $\chi$-colouring of a graph $G$. Its colour sequence can be written in a product form as $c_1^{\theta(c_1)}\circ c_2^{\theta(c_2)}\circ \ldots\circ c_\ell^{\theta(c_\ell)}=\prod_{i=1}^{\ell}c_i^{\theta(c_i)}$, where $\theta(c_i)$ is the number of vertices in $G$ having colour $c_i$. Then,
\begin{enumerate}\itemsep0mm
\item[(i)]~ the \textit{chromatic curling number} of $G$, denoted by $\cn(G)$, is defined as $\cn(G)=\max\{\theta(c_i): 1\le i\le \ell\}$.
\item[(ii)]~ the \textit{chromatic compound curling number} of $G$, denoted by $\cn(G)$, is defined as $\cnc(G)=\prod_{i=1}^{\ell}\theta(c_i)$. where $1\le i\le \ell$.
\end{enumerate}
\end{definition}

Since in $\chi^-$-colouring and $\chi^+$-colouring of $G$, the colours are interchanged with respect to certain conditions, we notice that the curling number and the compound curling number of $G$ with respect to both $\chi^+$-colouring and $\chi^-$-colouring are the same. In this section, we discuss the  chromatic curling number and  chromatic compound curling number of certain fundamental graphs. 

\noi The following result provides the chromatic curling number and the chromatic compound curling number of a path on $n$ vertices.

\begin{proposition}
The  chromatic curling number of $P_n$ is
$$\cn(P_n)=
\begin{cases}
\frac{n}{2}; & \text{$n$ is even};\\
\frac{n+1}{2}; & \text{$n$ is odd}.
\end{cases}$$
and
$$\cnc(P_n)=
\begin{cases}
\frac{n^2}{4}; & \text{$n$ is even};\\
\frac{n^2-1}{4}; & \text{$n$ is odd}.
\end{cases}$$
\end{proposition}
\begin{proof}~~
A chromatic colouring of a path consists of exactly $2$ colours, say $c_1$ and $c_2$. If $n$ is even, then $\theta(c_1)=\theta(c_2)=\frac{n}{2}$ and hence $\cn(P_n)=\frac{n}{2}$ and $\cnc(P_n)=\frac{n^2}{4}$. If $n$ is odd, then $\theta(c_1)=\frac{n+1}{2},\ \theta(c_2)=\frac{n-1}{2}$. Therefore, $\cn(P_n)=\frac{n+1}{2}$ and $\cnc(P_n)=\frac{n^2-1}{4}$.
\end{proof}

\noi The following theorem discusses the chromatic curling number and chromatic compound curling number of a cycle on $n$ vertices.

\begin{theorem}
The  chromatic curling number of the cycle $C_n$ is
$$\cn(C_n)=
\begin{cases}
\frac{n}{2}; & \text{$n$ is even};\\
\frac{n-1}{2}; & \text{$n$ is odd};
\end{cases}$$
and the chromatic compound curling number is
$$\cnc(C_n)=
\begin{cases}
\frac{n^2}{4}; & \text{$n$ is even};\\
\frac{(n-1)^2}{4}; & \text{$n$ is odd}.
\end{cases}$$
\end{theorem}
\begin{proof}~
A chromatic colouring of an even cycle consists of $2$ colours, say $c_1$ and $c_2$. In this case, we have $\theta(c_1)=\theta(c_2)=\frac{n}{2}$. Therefore, $\cn(C_n)=\frac{n}{2}$ and $\cnc(C_n)=\frac{n^2}{4}$. If $n$ is odd, then a chromatic colouring of $C_n$ must have three colours, say $c_1,c_2$ and $c_3$ such that $\theta(c_1)=\theta(c_2)=\frac{n-1}{2}$ and $\theta(c_3)=1$. Therefore,  $\cn(C_n)=\frac{n-1}{2}$ and $\cnc(C_n)=\frac{(n-1)^2}{4}$.
\end{proof}

A \textit{wheel graph} $W_{n+1}$ is a graph obtained by joining all vertices of a cycle $C_n$ to an external vertex, say $v$. This external vertex $v$ may be called the central vertex of $W_n$ and the cycle $C_n$ may be called the \textit{rim} of $W_{n+1}$. The following result discusses the chromatic curling number and chromatic compound curling number of a wheel graph $W_{n+1}$. 

\begin{theorem}
The  chromatic curling number and  chromatic compound curling number of the wheel graph $W_{n+1}$ equal to those of its rim.
\end{theorem}
\begin{proof}~
The proof is straight forward. (See Figure \ref{fig:e-wl})
\end{proof}

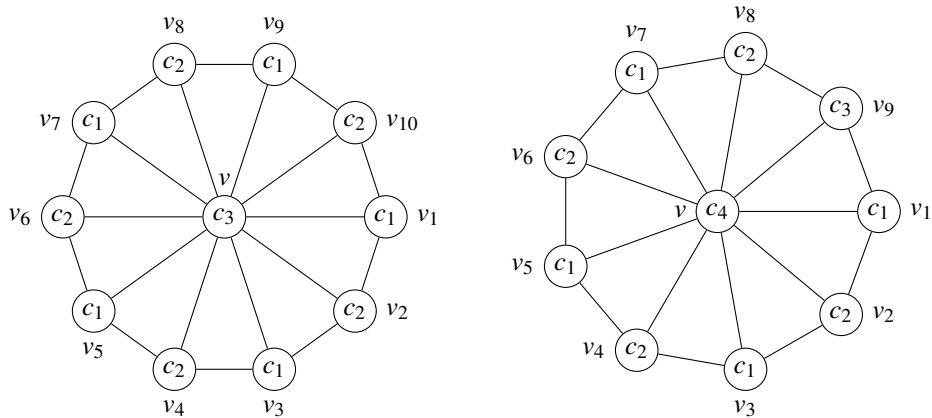
\begin{figure}[h!]
\begin{center}
\begin{tikzpicture}[scale=0.85] 
\vertex (v1) at (0:2.5) [label=right:$v_{1}$]{$c_1$};
\vertex (v2) at (324:2.5) [label=right:$v_{2}$]{$c_2$};
\vertex (v3) at (288:2.5) [label=below:$v_{3}$]{$c_1$};
\vertex (v4) at (252:2.5) [label=below:$v_{4}$]{$c_2$};
\vertex (v5) at (216:2.5) [label=below:$v_{5}$]{$c_1$};
\vertex (v6) at (180:2.5) [label=left:$v_{6}$]{$c_2$};
\vertex (v7) at (144:2.5) [label=left:$v_{7}$]{$c_1$};
\vertex (v8) at (108:2.5) [label=above:$v_{8}$]{$c_2$};
\vertex (v9) at (72:2.5) [label=above:$v_{9}$]{$c_1$};
\vertex (v10) at (36:2.5) [label=right:$v_{10}$]{$c_2$};
\vertex (v) at (0:0) [label=above:$v$]{$c_3$};
\path 
(v1) edge (v2)
(v1) edge (v)
(v2) edge (v3)
(v2) edge (v)
(v3) edge (v4)
(v3) edge (v)
(v4) edge (v5)
(v4) edge (v)
(v5) edge (v6)
(v5) edge (v)
(v6) edge (v7)
(v6) edge (v)
(v7) edge (v8)
(v7) edge (v)
(v8) edge (v9)
(v8) edge (v)
(v9) edge (v10)
(v9) edge (v)
(v10) edge (v1)
(v10) edge (v)
;
\end{tikzpicture}
\qquad 
\begin{tikzpicture}[scale=0.85] 
\vertex (v1) at (0:2.5) [label=right:$v_{1}$]{$c_1$};
\vertex (v2) at (320:2.5) [label=right:$v_{2}$]{$c_2$};
\vertex (v3) at (280:2.5) [label=below:$v_{3}$]{$c_1$};
\vertex (v4) at (240:2.5) [label=left:$v_{4}$]{$c_2$};
\vertex (v5) at (200:2.5) [label=left:$v_{5}$]{$c_1$};
\vertex (v6) at (160:2.5) [label=left:$v_{6}$]{$c_2$};
\vertex (v7) at (120:2.5) [label=above:$v_{7}$]{$c_1$};
\vertex (v8) at (80:2.5) [label=above:$v_{8}$]{$c_2$};
\vertex (v9) at (40:2.5) [label=right:$v_{9}$]{$c_3$};
\vertex (v) at (0:0) [label=left:$v$]{$c_4$};
\path 
(v1) edge (v2)
(v2) edge (v3)
(v3) edge (v4)
(v4) edge (v5)
(v5) edge (v6)
(v6) edge (v7)
(v7) edge (v8)
(v8) edge (v9)
(v9) edge (v1)
(v1) edge (v)
(v2) edge (v)
(v3) edge (v)
(v4) edge (v)
(v5) edge (v)
(v6) edge (v)
(v7) edge (v)
(v8) edge (v)
(v9) edge (v)
;
\end{tikzpicture}
\end{center}
\caption{\small A chromatic colouring of wheel graphs}\label{fig:e-wl}
\end{figure}

A \textit{double wheel graph} $DW_n$ is a graph obtained by joining all vertices of the two disjoint cycles to an external vertex. That is, $DW_n=2C_n+K_1$. The following result discusses the chromatic curling number and chromatic compound curling number of a double wheel graph.

\begin{theorem}\label{Thm-1a}
For a double wheel graph $DW_n=2C_n+K_1$, we have 
$$\cn(DW_n)=\begin{cases}
n; & \text{if $n$ is even}\\
n-1; & \text{if $n$ is odd}
\end{cases}$$ 
and 
$$\cnc(DW_n)=\begin{cases}
n^2; & \text{if $n$ is even}\\
2(n-1)^2; & \text{if $n$ is odd}
\end{cases}$$.
\end{theorem}
\begin{proof}~
In $DW_n$, the corresponding vertices of both cycles (rims) can assume same colour. Hence, we have $\theta_{DW_n}(c_i)=2\cdot \theta_{C_n}(c_i)$ for all $1\le i\le \chi(G)$.  Hence the proof is straight forward. See Figure \ref{fig:e-dwl} for illustration.
\end{proof}

\begin{figure}[h!]
\centering
\begin{tikzpicture}[auto,node distance=1.75cm,
thick,main node/.style={circle,draw,font=\sffamily\Large\bfseries}]
\vertex (v1) at (0:1.75) []{$c_1$};
\vertex (v2) at (320:1.75) []{$c_2$};
\vertex (v3) at (280:1.75) []{$c_1$};
\vertex (v4) at (240:1.75) []{$c_2$};
\vertex (v5) at (200:1.75) []{$c_1$};
\vertex (v6) at (160:1.75) []{$c_2$};
\vertex (v7) at (120:1.75) []{$c_1$};
\vertex (v8) at (80:1.75) []{$c_2$};
\vertex (v9) at (40:1.75) []{$c_3$};
\vertex (v) at (0:0) []{$c_{4}$};
\vertex (u1) at (0:3) []{$c_1$};
\vertex (u2) at (320:3) []{$c_2$};
\vertex (u3) at (280:3) []{$c_1$};
\vertex (u4) at (240:3) []{$c_2$};
\vertex (u5) at (200:3) []{$c_1$};
\vertex (u6) at (160:3) []{$c_2$};
\vertex (u7) at (120:3) []{$c_1$};
\vertex (u8) at (80:3) []{$c_2$};
\vertex (u9) at (40:3) []{$c_3$};
\path 
(v1) edge (v2)
(v2) edge (v3)
(v3) edge (v4)
(v4) edge (v5)
(v5) edge (v6)
(v6) edge (v7)
(v7) edge (v8)
(v8) edge (v9)
(v9) edge (v1)
(u1) edge (u2)
(u2) edge (u3)
(u3) edge (u4)
(u4) edge (u5)
(u5) edge (u6)
(u6) edge (u7)
(u7) edge (u8)
(u8) edge (u9)
(u9) edge (u1)
(v1) edge (v)
(v2) edge (v)
(v3) edge (v)
(v4) edge (v)
(v5) edge (v)
(v6) edge (v)
(v7) edge (v)
(v8) edge (v)
(v9) edge (v)
(u1) edge [bend right] (v)
(u2) edge [bend right] (v)
(u3) edge [bend right] (v)
(u4) edge [bend right] (v)
(u5) edge [bend right] (v)
(u6) edge [bend right] (v)
(u7) edge [bend right] (v)
(u8) edge [bend right] (v)
(u9) edge [bend right] (v)
;
\end{tikzpicture}
\caption{\small Chromatic colouring of a double wheel graph}\label{fig:e-dwl}
\end{figure}
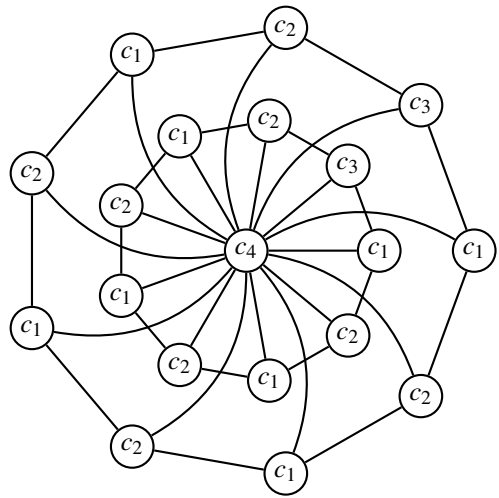 

A \textit{helm graph} $H_n$ is a graph obtained by attaching a pendant edge to every vertex of the rim $C_n$ of a wheel graph $W_n$. The following result provides the chromatic curling number and chromatic compound curling number of the helm graphs.

\begin{theorem}\label{Thm-2}
For the helm graph $H_n$, we have 
$\cn(H_n)=n+1$ and the  chromatic compound curling number is
$$\cnc(H_n)=
\begin{cases}
\frac{n^3+n^2}{4}; & \text{if $n$ is even}\\
\frac{(n-1)^2(n+1)}{4}; & \text{if $n$ is odd}.
\end{cases}$$
\end{theorem}
\begin{proof}~
Note that in a chromatic colouring of a helm graph $H_n$, the central vertex and the $n$ pendant vertices can assume the same colour $c_1$. Then, $\theta(c_1)=n+1$. Moreover, if $n$ is even, $\theta(c_2)=\theta(c_3)=\frac{n}{2}$ and if $n$ is odd, then $\theta(c_2)=\theta(c_3)=\frac{n-1}{2},\ \theta(c_3)=1$. Then, the result is immediate. 
\end{proof}

\begin{figure}[h!]
\begin{center}
\begin{tikzpicture}[scale=0.85] 
\vertex (v1) at (0:1.75) []{$c_2$};
\vertex (v2) at (315:1.75) []{$c_3$};
\vertex (v3) at (270:1.75) []{$c_2$};
\vertex (v4) at (225:1.75) []{$c_3$};
\vertex (v5) at (180:1.75) []{$c_2$};
\vertex (v6) at (135:1.75) []{$c_3$};
\vertex (v7) at (90:1.75) []{$c_2$};
\vertex (v8) at (45:1.75) []{$c_3$};
\vertex (v) at (0:0) []{$c_1$};
\vertex (u1) at (0:3) []{$c_1$};
\vertex (u2) at (315:3) []{$c_1$};
\vertex (u3) at (270:3) []{$c_1$};
\vertex (u4) at (225:3) []{$c_1$};
\vertex (u5) at (180:3) []{$c_1$};
\vertex (u6) at (135:3) []{$c_1$};
\vertex (u7) at (90:3) []{$c_1$};
\vertex (u8) at (45:3) []{$c_1$};
\path 
(v1) edge (v2)
(v1) edge (v8)
(v1) edge (v)
(v2) edge (v3)
(v2) edge (v)
(v3) edge (v4)
(v3) edge (v)
(v4) edge (v5)
(v4) edge (v)
(v5) edge (v6)
(v5) edge (v)
(v6) edge (v7)
(v6) edge (v)
(v7) edge (v8)
(v7) edge (v)
(v8) edge (v)
(v1) edge (u1)
(v2) edge (u2)
(v3) edge (u3)
(v4) edge (u4)
(v5) edge (u5)
(v6) edge (u6)
(v7) edge (u7)
(v8) edge (u8)
;
\end{tikzpicture}
\qquad 
\begin{tikzpicture}[scale=0.85] 
\vertex (v1) at (0:1.75) []{$c_2$};
\vertex (v2) at (320:1.75) []{$c_3$};
\vertex (v3) at (280:1.75) []{$c_2$};
\vertex (v4) at (240:1.75) []{$c_3$};
\vertex (v5) at (200:1.75) []{$c_2$};
\vertex (v6) at (160:1.75) []{$c_3$};
\vertex (v7) at (120:1.75) []{$c_2$};
\vertex (v8) at (80:1.75) []{$c_3$};
\vertex (v9) at (40:1.75) []{$c_4$};
\vertex (v) at (0:0) []{$c_1$};
\vertex (u1) at (0:3) []{$c_1$};
\vertex (u2) at (320:3) []{$c_1$};
\vertex (u3) at (280:3) []{$c_1$};
\vertex (u4) at (240:3) []{$c_1$};
\vertex (u5) at (200:3) []{$c_1$};
\vertex (u6) at (160:3) []{$c_1$};
\vertex (u7) at (120:3) []{$c_1$};
\vertex (u8) at (80:3) []{$c_1$};
\vertex (u9) at (40:3) []{$c_1$};
\path 
(v1) edge (v2)
(v2) edge (v3)
(v3) edge (v4)
(v4) edge (v5)
(v5) edge (v6)
(v6) edge (v7)
(v7) edge (v8)
(v8) edge (v9)
(v9) edge (v1)
(v1) edge (v)
(v2) edge (v)
(v3) edge (v)
(v4) edge (v)
(v5) edge (v)
(v6) edge (v)
(v7) edge (v)
(v8) edge (v)
(v9) edge (v)
(v1) edge (u1)
(v2) edge (u2)
(v3) edge (u3)
(v4) edge (u4)
(v5) edge (u5)
(v6) edge (u6)
(v7) edge (u7)
(v8) edge (u8)
(v9) edge (u9)
;
\end{tikzpicture}
\end{center}
\caption{\small Chromatic colouring of helm graphs}\label{fig:e-hl}
\end{figure}
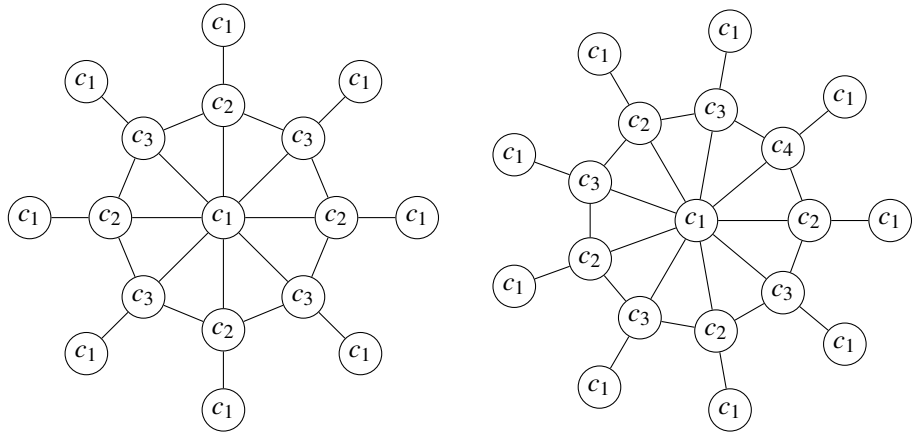


A \textit{closed helm graph} $CH_n$ is a graph obtained from the helm graph $H_n$, by joining its the pendant vertices such that these vertices induce a cycle $C_n$. Then, we can verify that a chromatic colouring we defined for $H_n$ in the above result will be a chromatic colouring of $CH_n$ as well. Then, we have the following straight forward result. 

\begin{theorem}\label{Thm-3}
For a closed helm graph $CH_n$, we have 
$$\cn(H_n)=
\begin{cases}
n; & \text{if $n$ is even}\\
n-1; & \text{if $n$ is odd}.
\end{cases}$$ and 
$$\cnc(H_n)=
\begin{cases}
n^2; & \text{if $n$ is even}\\
2(n-1)^2; & \text{if $n$ is odd}.
\end{cases}$$
\end{theorem}
\begin{proof}~
The proof is similar to that of Theorem \ref{Thm-1a}. The  colouring of closed helm graphs can be seen in Figure \ref{fig:e-chl}.
\end{proof}

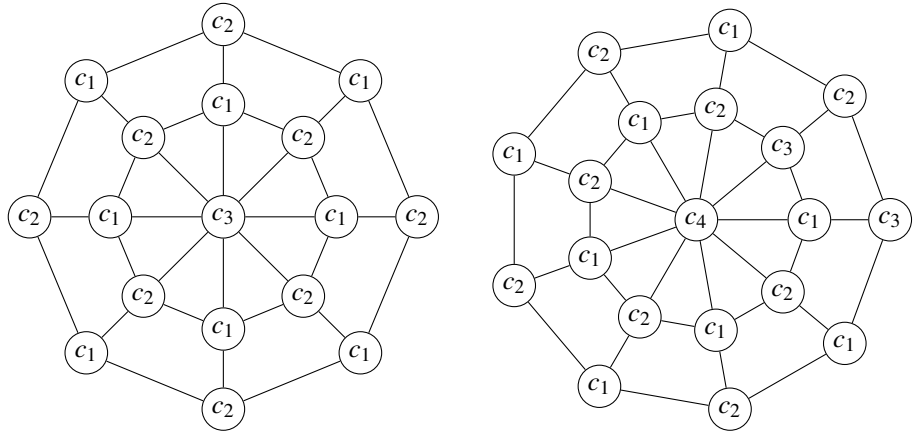
\begin{figure}[h!]
\begin{center}
\begin{tikzpicture}[scale=0.85] 
\vertex (v1) at (0:1.75) []{$c_1$};
\vertex (v2) at (315:1.75) []{$c_2$};
\vertex (v3) at (270:1.75) []{$c_1$};
\vertex (v4) at (225:1.75) []{$c_2$};
\vertex (v5) at (180:1.75) []{$c_1$};
\vertex (v6) at (135:1.75) []{$c_2$};
\vertex (v7) at (90:1.75) []{$c_1$};
\vertex (v8) at (45:1.75) []{$c_2$};
\vertex (v) at (0:0) []{$c_3$};
\vertex (u1) at (0:3) []{$c_2$};
\vertex (u2) at (315:3) []{$c_1$};
\vertex (u3) at (270:3) []{$c_2$};
\vertex (u4) at (225:3) []{$c_1$};
\vertex (u5) at (180:3) []{$c_2$};
\vertex (u6) at (135:3) []{$c_1$};
\vertex (u7) at (90:3) []{$c_2$};
\vertex (u8) at (45:3) []{$c_1$};
\path 
(v1) edge (v2)
(v1) edge (v8)
(v1) edge (v)
(v2) edge (v3)
(v2) edge (v)
(v3) edge (v4)
(v3) edge (v)
(v4) edge (v5)
(v4) edge (v)
(v5) edge (v6)
(v5) edge (v)
(v6) edge (v7)
(v6) edge (v)
(v7) edge (v8)
(v7) edge (v)
(v8) edge (v)
(v1) edge (u1)
(v2) edge (u2)
(v3) edge (u3)
(v4) edge (u4)
(v5) edge (u5)
(v6) edge (u6)
(v7) edge (u7)
(v8) edge (u8)
(u1) edge (u2)
(u2) edge (u3)
(u3) edge (u4)
(u4) edge (u5)
(u5) edge (u6)
(u6) edge (u7)
(u7) edge (u8)
(u8) edge (u1)
;
\end{tikzpicture}
\qquad 
\begin{tikzpicture}[scale=0.85] 
\vertex (v1) at (0:1.75) []{$c_1$};
\vertex (v2) at (320:1.75) []{$c_2$};
\vertex (v3) at (280:1.75) []{$c_1$};
\vertex (v4) at (240:1.75) []{$c_2$};
\vertex (v5) at (200:1.75) []{$c_1$};
\vertex (v6) at (160:1.75) []{$c_2$};
\vertex (v7) at (120:1.75) []{$c_1$};
\vertex (v8) at (80:1.75) []{$c_2$};
\vertex (v9) at (40:1.75) []{$c_3$};
\vertex (v) at (0:0) []{$c_4$};
\vertex (u1) at (0:3) []{$c_3$};
\vertex (u2) at (320:3) []{$c_1$};
\vertex (u3) at (280:3) []{$c_2$};
\vertex (u4) at (240:3) []{$c_1$};
\vertex (u5) at (200:3) []{$c_2$};
\vertex (u6) at (160:3) []{$c_1$};
\vertex (u7) at (120:3) []{$c_2$};
\vertex (u8) at (80:3) []{$c_1$};
\vertex (u9) at (40:3) []{$c_2$};
\path 
(v1) edge (v2)
(v2) edge (v3)
(v3) edge (v4)
(v4) edge (v5)
(v5) edge (v6)
(v6) edge (v7)
(v7) edge (v8)
(v8) edge (v9)
(v9) edge (v1)
(v1) edge (v)
(v2) edge (v)
(v3) edge (v)
(v4) edge (v)
(v5) edge (v)
(v6) edge (v)
(v7) edge (v)
(v8) edge (v)
(v9) edge (v)
(v1) edge (u1)
(v2) edge (u2)
(v3) edge (u3)
(v4) edge (u4)
(v5) edge (u5)
(v6) edge (u6)
(v7) edge (u7)
(v8) edge (u8)
(v9) edge (u9)
(u1) edge (u2)
(u2) edge (u3)
(u3) edge (u4)
(u4) edge (u5)
(u5) edge (u6)
(u6) edge (u7)
(u7) edge (u8)
(u8) edge (u9)
(u9) edge (u1)
;
\end{tikzpicture}
\end{center}
\caption{\small Chromatic colouring of closed helm graphs}\label{fig:e-chl}
\end{figure}

A \textit{flower graph} $F_n$ is a graph which is obtained by joining the pendant vertices of a helm graph $H_n$ to its central vertex. The  chromatic curling numbers of the flower graph $F_n$ is determined in the following theorem. 

\begin{theorem}\label{Thm-4}
For a flower graph $F_n$, we have $\cn(F_n)=n$ and $$\cnc(F_n)=
\cnc(F_n)= \begin{cases}
\frac{n^3}{4}; & \text{if $n$ is even}\\
\frac{n(n-1)^2}{4}; & \text{if $n$ is odd}.
\end{cases}$$
\end{theorem}
\begin{proof}~
All outer (independent) vertices of $F_n$ can be given the colour $c_1$ irrespective of whether $n$ is even or odd. If $n$ is even, the vertices of the rim can be given the colours $c_2$ and $c_3$ alternatively and the central vertex can be given the colour $c_4$. Therefore, $\theta(c_1)=n, \theta(c_2)=\theta(c_3)=\frac{n}{2}$ and $\theta(c_4)=1$. Hence, in this case, $\cn(F_n)=n$ and $\cnc(F_n)=\frac{n^3}{4}$.

If $n$ is odd, the vertices, except one, of the rim can be given the colours $c_2$ and $c_3$ alternatively, the above mentioned single vertex can be given the colour $c_4$ and the central vertex can be given the colour $c_5$. Therefore, $\theta(c_1)=n, \theta(c_2)=\theta(c_3)=\frac{n-1}{2}$ and $\theta(c_4)=\theta(c_5)=1$. Hence, in this case, $\cn(F_n)=n$ and $\cnc(F_n)=\frac{n(n-1)^2}{4}$.
\end{proof}

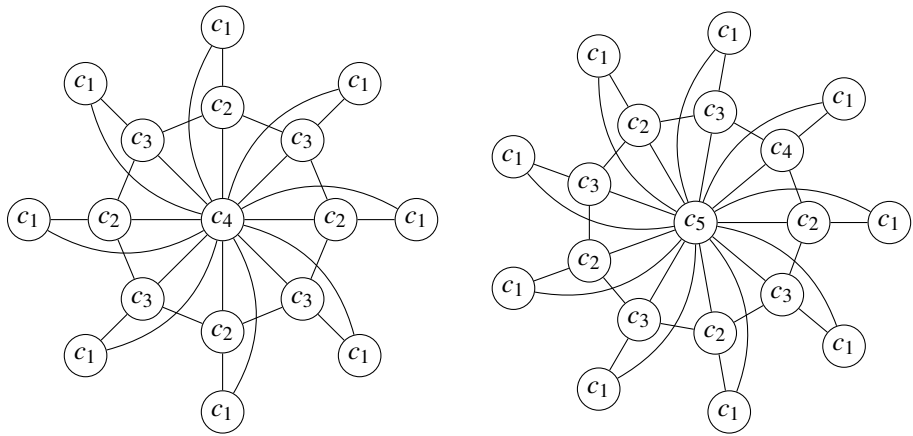
\begin{figure}[h!]
\begin{center}
\begin{tikzpicture}[scale=0.85] 
\vertex (v1) at (0:1.75) []{$c_2$};
\vertex (v2) at (315:1.75) []{$c_3$};
\vertex (v3) at (270:1.75) []{$c_2$};
\vertex (v4) at (225:1.75) []{$c_3$};
\vertex (v5) at (180:1.75) []{$c_2$};
\vertex (v6) at (135:1.75) []{$c_3$};
\vertex (v7) at (90:1.75) []{$c_2$};
\vertex (v8) at (45:1.75) []{$c_3$};
\vertex (v) at (0:0) []{$c_4$};
\vertex (u1) at (0:3) []{$c_1$};
\vertex (u2) at (315:3) []{$c_1$};
\vertex (u3) at (270:3) []{$c_1$};
\vertex (u4) at (225:3) []{$c_1$};
\vertex (u5) at (180:3) []{$c_1$};
\vertex (u6) at (135:3) []{$c_1$};
\vertex (u7) at (90:3) []{$c_1$};
\vertex (u8) at (45:3) []{$c_1$};
\path 
(v1) edge (v2)
(v1) edge (v8)
(v1) edge (v)
(v2) edge (v3)
(v2) edge (v)
(v3) edge (v4)
(v3) edge (v)
(v4) edge (v5)
(v4) edge (v)
(v5) edge (v6)
(v5) edge (v)
(v6) edge (v7)
(v6) edge (v)
(v7) edge (v8)
(v7) edge (v)
(v8) edge (v)
(v1) edge (u1)
(v2) edge (u2)
(v3) edge (u3)
(v4) edge (u4)
(v5) edge (u5)
(v6) edge (u6)
(v7) edge (u7)
(v8) edge (u8)
(u1) edge [bend right] (v)
(u2) edge [bend right] (v)
(u3) edge [bend right] (v)
(u4) edge [bend right] (v)
(u5) edge [bend right] (v)
(u6) edge [bend right] (v)
(u7) edge [bend right] (v)
(u8) edge [bend right] (v)
;
\end{tikzpicture}
\qquad 
\begin{tikzpicture}[scale=0.85] 
\vertex (v1) at (0:1.75) []{$c_2$};
\vertex (v2) at (320:1.75) []{$c_3$};
\vertex (v3) at (280:1.75) []{$c_2$};
\vertex (v4) at (240:1.75) []{$c_3$};
\vertex (v5) at (200:1.75) []{$c_2$};
\vertex (v6) at (160:1.75) []{$c_3$};
\vertex (v7) at (120:1.75) []{$c_2$};
\vertex (v8) at (80:1.75) []{$c_3$};
\vertex (v9) at (40:1.75) []{$c_4$};
\vertex (v) at (0:0) []{$c_{5}$};
\vertex (u1) at (0:3) []{$c_1$};
\vertex (u2) at (320:3) []{$c_1$};
\vertex (u3) at (280:3) []{$c_1$};
\vertex (u4) at (240:3) []{$c_1$};
\vertex (u5) at (200:3) []{$c_1$};
\vertex (u6) at (160:3) []{$c_1$};
\vertex (u7) at (120:3) []{$c_1$};
\vertex (u8) at (80:3) []{$c_1$};
\vertex (u9) at (40:3) []{$c_1$};
\path 
(v1) edge (v2)
(v2) edge (v3)
(v3) edge (v4)
(v4) edge (v5)
(v5) edge (v6)
(v6) edge (v7)
(v7) edge (v8)
(v8) edge (v9)
(v9) edge (v1)
(v1) edge (v)
(v2) edge (v)
(v3) edge (v)
(v4) edge (v)
(v5) edge (v)
(v6) edge (v)
(v7) edge (v)
(v8) edge (v)
(v9) edge (v)
(v1) edge (u1)
(v2) edge (u2)
(v3) edge (u3)
(v4) edge (u4)
(v5) edge (u5)
(v6) edge (u6)
(v7) edge (u7)
(v8) edge (u8)
(v9) edge (u9)
(u1) edge [bend right] (v)
(u2) edge [bend right] (v)
(u3) edge [bend right] (v)
(u4) edge [bend right] (v)
(u5) edge [bend right] (v)
(u6) edge [bend right] (v)
(u7) edge [bend right] (v)
(u8) edge [bend right] (v)
(u9) edge [bend right] (v)
;
\end{tikzpicture}
\end{center}
\caption{\small Chromatic colouring of flower graphs}\label{fig:e-fl}
\end{figure}


A \textit{djembe graph} $Dj_n$ is the graph obtained by joining the corresponding vertices two cycles of the same order $n$ and joining all the vertices of the two cycles to an external vertex (see \cite{NKS1}). That is, $Dj_n=(C_n\Box P_2)+K_1$. This external vertex may be called the central vertex of the djembe graph. The following result discusses the equitable colouring parameters of djembe graphs.

\begin{theorem}
For a djembe graph $Dj_n=(C_n\Box C_n)+K_1$, we have
$$\cn(Dj_n)= \begin{cases}
n; & \text{if $n$ is even}\\
n-1 & \text{if $n$ is odd}.
\end{cases}$$ 
and
$$\cnc(Dj_n)=\begin{cases}
n^2; & \text{if $n$ is even}\\
2(n-1)^2; & \text{if $n$ is odd}.
\end{cases}$$
\end{theorem}
\begin{proof}~
The proof is similar to that of Theorem \ref{Thm-1a} and Theorem \ref{Thm-3}. See Figure \ref{fig:e-djl} for illustration of chromatic colouring of djembe graphs.
\end{proof}

\begin{figure}[h!]
\begin{center}
\begin{tikzpicture}[scale=0.85] 
\vertex (v1) at (0:1.75) []{$c_1$};
\vertex (v2) at (320:1.75) []{$c_2$};
\vertex (v3) at (280:1.75) []{$c_1$};
\vertex (v4) at (240:1.75) []{$c_2$};
\vertex (v5) at (200:1.75) []{$c_1$};
\vertex (v6) at (160:1.75) []{$c_2$};
\vertex (v7) at (120:1.75) []{$c_1$};
\vertex (v8) at (80:1.75) []{$c_2$};
\vertex (v9) at (40:1.75) []{$c_3$};
\vertex (v) at (0:0) []{$c_{4}$};
\vertex (u1) at (0:3) []{$c_3$};
\vertex (u2) at (320:3) []{$c_1$};
\vertex (u3) at (280:3) []{$c_2$};
\vertex (u4) at (240:3) []{$c_1$};
\vertex (u5) at (200:3) []{$c_2$};
\vertex (u6) at (160:3) []{$c_1$};
\vertex (u7) at (120:3) []{$c_2$};
\vertex (u8) at (80:3) []{$c_1$};
\vertex (u9) at (40:3) []{$c_2$};
\path 
(v1) edge (v2)
(v2) edge (v3)
(v3) edge (v4)
(v4) edge (v5)
(v5) edge (v6)
(v6) edge (v7)
(v7) edge (v8)
(v8) edge (v9)
(v9) edge (v1)
(u1) edge (u2)
(u2) edge (u3)
(u3) edge (u4)
(u4) edge (u5)
(u5) edge (u6)
(u6) edge (u7)
(u7) edge (u8)
(u8) edge (u9)
(u9) edge (u1)
(v1) edge (v)
(v2) edge (v)
(v3) edge (v)
(v4) edge (v)
(v5) edge (v)
(v6) edge (v)
(v7) edge (v)
(v8) edge (v)
(v9) edge (v)
(u1) edge [bend right] (v)
(u2) edge [bend right] (v)
(u3) edge [bend right] (v)
(u4) edge [bend right] (v)
(u5) edge [bend right] (v)
(u6) edge [bend right] (v)
(u7) edge [bend right] (v)
(u8) edge [bend right] (v)
(u9) edge [bend right] (v)
(u1) edge (v1)
(u2) edge (v2)
(u3) edge (v3)
(u4) edge (v4)
(u5) edge (v5)
(u6) edge (v6)
(u7) edge (v7)
(u8) edge (v8)
(u9) edge (v9)
;
\end{tikzpicture}
\qquad 
\begin{tikzpicture}[scale=0.85] 
\vertex (v1) at (0:1.75) []{$c_1$};
\vertex (v2) at (315:1.75) []{$c_2$};
\vertex (v3) at (270:1.75) []{$c_1$};
\vertex (v4) at (225:1.75) []{$c_2$};
\vertex (v5) at (180:1.75) []{$c_1$};
\vertex (v6) at (135:1.75) []{$c_2$};
\vertex (v7) at (90:1.75) []{$c_1$};
\vertex (v8) at (45:1.75) []{$c_2$};
\vertex (v) at (0:0) []{$c_{3}$};
\vertex (u1) at (0:3) []{$c_2$};
\vertex (u2) at (315:3) []{$c_1$};
\vertex (u3) at (270:3) []{$c_2$};
\vertex (u4) at (225:3) []{$c_1$};
\vertex (u5) at (180:3) []{$c_2$};
\vertex (u6) at (135:3) []{$c_1$};
\vertex (u7) at (90:3) []{$c_2$};
\vertex (u8) at (45:3) []{$c_1$};
\path 
(v1) edge (v2)
(v2) edge (v3)
(v3) edge (v4)
(v4) edge (v5)
(v5) edge (v6)
(v6) edge (v7)
(v7) edge (v8)
(v8) edge (v1)
(u1) edge (u2)
(u2) edge (u3)
(u3) edge (u4)
(u4) edge (u5)
(u5) edge (u6)
(u6) edge (u7)
(u7) edge (u8)
(u8) edge (u1)
(v1) edge (v)
(v2) edge (v)
(v3) edge (v)
(v4) edge (v)
(v5) edge (v)
(v6) edge (v)
(v7) edge (v)
(v8) edge (v)
(u1) edge [bend right] (v)
(u2) edge [bend right] (v)
(u3) edge [bend right] (v)
(u4) edge [bend right] (v)
(u5) edge [bend right] (v)
(u6) edge [bend right] (v)
(u7) edge [bend right] (v)
(u8) edge [bend right] (v)
(u1) edge (v1)
(u2) edge (v2)
(u3) edge (v3)
(u4) edge (v4)
(u5) edge (v5)
(u6) edge (v6)
(u7) edge (v7)
(u8) edge (v8)
;
\end{tikzpicture}
\end{center}
\caption{\small Chromatic colouring of a djembe graph}\label{fig:e-djl}
\end{figure}
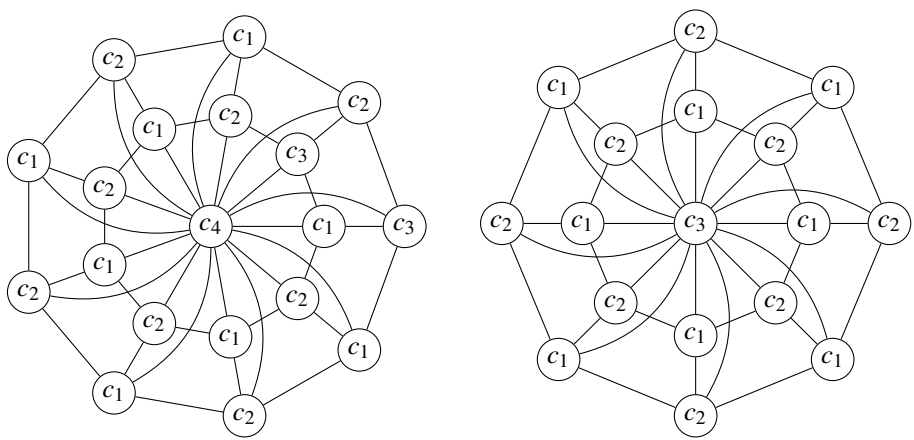 

A \textit{sunflower graph} $SF_n$ is a graph obtained by replacing each edge of the rim of a wheel graph $W_n$ by a triangle such that two triangles share a common vertex if and only if the corresponding edges in $W_n$ are adjacent in $W_n$. The following theorem determines the chromatic curling number and chromatic compound curling number of the sunflower graph $SF_n$.

\begin{theorem}\label{Thm-5}
For the sunflower graph $SF_n$, we have $\cn(SF_n)=n+1$ and $$\cnc(SF_n)=\begin{cases}
\frac{(n+1)n^2}{4}; & \text{if $n$ is even}\\
\frac{(n+1)(n-1)^2}{4}; & \text{if $n$ is odd}.
\end{cases}$$
\end{theorem}
\begin{proof}~
The proof is similar to that of Theorem \ref{Thm-2} (see Figure \ref{fig:e-sfl} for illustration). 
\end{proof}

\begin{figure}[h!]
\begin{center}
\begin{tikzpicture}[scale=0.85] 
\vertex (v1) at (0:1.75) []{$c_2$};
\vertex (v2) at (315:1.75) []{$c_3$};
\vertex (v3) at (270:1.75) []{$c_2$};
\vertex (v4) at (225:1.75) []{$c_3$};
\vertex (v5) at (180:1.75) []{$c_2$};
\vertex (v6) at (135:1.75) []{$c_3$};
\vertex (v7) at (90:1.75) []{$c_2$};
\vertex (v8) at (45:1.75) []{$c_3$};
\vertex (v) at (0:0) []{$c_{1}$};
\vertex (u1) at (337.5:3) []{$c_1$};
\vertex (u2) at (292.5:3) []{$c_1$};
\vertex (u3) at (247.5:3) []{$c_1$};
\vertex (u4) at (202.5:3) []{$c_1$};
\vertex (u5) at (157.5:3) []{$c_1$};
\vertex (u6) at (112.5:3) []{$c_1$};
\vertex (u7) at (67.5:3) []{$c_1$};
\vertex (u8) at (22.5:3) []{$c_1$};
\path 
(v1) edge (v2)
(v2) edge (v3)
(v3) edge (v4)
(v4) edge (v5)
(v5) edge (v6)
(v6) edge (v7)
(v7) edge (v8)
(v8) edge (v1)
(v1) edge (v)
(v2) edge (v)
(v3) edge (v)
(v4) edge (v)
(v5) edge (v)
(v6) edge (v)
(v7) edge (v)
(v8) edge (v)
(u1) edge (v1)
(u1) edge (v2)
(u2) edge (v2)
(u2) edge (v3)
(u3) edge (v3)
(u3) edge (v4)
(u4) edge (v4)
(u4) edge (v5)
(u5) edge (v5)
(u5) edge (v6)
(u6) edge (v6)
(u6) edge (v7)
(u7) edge (v7)
(u7) edge (v8)
(u8) edge (v8)
(u8) edge (v1)
;
\end{tikzpicture}
\qquad 
\begin{tikzpicture}[scale=0.85] 
\vertex (v1) at (0:1.75) []{$c_2$};
\vertex (v2) at (320:1.75) []{$c_3$};
\vertex (v3) at (280:1.75) []{$c_2$};
\vertex (v4) at (240:1.75) []{$c_3$};
\vertex (v5) at (200:1.75) []{$c_2$};
\vertex (v6) at (160:1.75) []{$c_3$};
\vertex (v7) at (120:1.75) []{$c_2$};
\vertex (v8) at (80:1.75) []{$c_3$};
\vertex (v9) at (40:1.75) []{$c_4$};
\vertex (v) at (0:0) []{$c_{1}$};
\vertex (u1) at (20:3) []{$c_1$};
\vertex (u2) at (340:3) []{$c_1$};
\vertex (u3) at (300:3) []{$c_1$};
\vertex (u4) at (260:3) []{$c_1$};
\vertex (u5) at (220:3) []{$c_1$};
\vertex (u6) at (180:3) []{$c_1$};
\vertex (u7) at (140:3) []{$c_1$};
\vertex (u8) at (100:3) []{$c_1$};
\vertex (u9) at (60:3) []{$c_1$};
\path 
(v1) edge (v2)
(v2) edge (v3)
(v3) edge (v4)
(v4) edge (v5)
(v5) edge (v6)
(v6) edge (v7)
(v7) edge (v8)
(v8) edge (v9)
(v9) edge (v1)
(v1) edge (v)
(v2) edge (v)
(v3) edge (v)
(v4) edge (v)
(v5) edge (v)
(v6) edge (v)
(v7) edge (v)
(v8) edge (v)
(v9) edge (v)
(u1) edge (v1)
(u1) edge (v9)
(u2) edge (v1)
(u2) edge (v2)
(u3) edge (v2)
(u3) edge (v3)
(u4) edge (v3)
(u4) edge (v4)
(u5) edge (v4)
(u5) edge (v5)
(u6) edge (v5)
(u6) edge (v6)
(u7) edge (v6)
(u7) edge (v7)
(u8) edge (v7)
(u8) edge (v8)
(u9) edge (v8)
(u9) edge (v9)
;
\end{tikzpicture}
\end{center}
\caption{\small Chromatic colouring of a sunflower graph}\label{fig:e-sfl}
\end{figure}
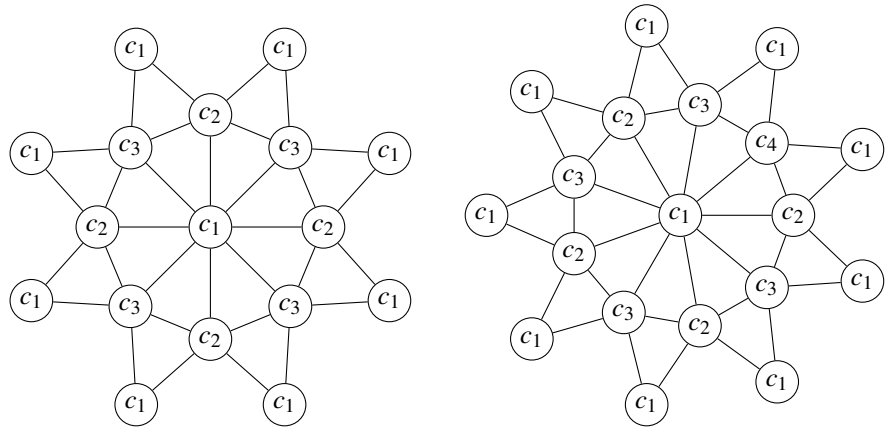

A \textit{closed sunflower graph} $CSF_n$ is the graph obtained by joining the independent vertices of a sunflower graph $SF_n$, which are not adjacent to its central vertex so that these vertices induces a cycle on $n$ vertices. The following result provides the equitable colouring parameters of a closed sunflower graph.

\begin{theorem}\label{Thm-6}
For the sunflower graph $CSF_n$, we have 
$$\cn(CSF_n)=
\begin{cases}
\frac{n+2}{2}; & \text{if $n$ is even}\\
\frac{n+1}{2}; & \text{if $n$ is odd}.
\end{cases}$$ 
and
$$\cnc(CSF_n)=
\begin{cases}
\frac{n^3(n+2)}{16}; & \text{if $n$ is even}\\
\frac{(n+1)^3(n-1)}{16}; & \text{if $n$ is odd}.
\end{cases}$$
\end{theorem}
\begin{proof}~
The chromatic colouring of a closed sunflower graph consists of $4$ colours, say $c_1,c_2,c_3,c_4$. If $n$ is even, the vertices of the outer cycle of $CSF_n$ can be given the colours $c_1$ and $c_2$ alternatively and the inner cycle of $CSF_n$ can be given the colours $c_3$ and $c_4$ alternatively. The central vertex can be given the colour $c_1$ (see the second graph Figure \ref{fig:e-csfl} for illustration). That is, $\theta(c_1)=\frac{n}{2}+1$ and $\theta(c_2)=\theta(c_3)=\theta(c_4)=\frac{n}{2}$. Therefore, $\cn(CSF_n)=\frac{n+2}{2}$ and $\cnc(CSF_n)=\frac{n^3(n+2)}{16}$.

If $n$ is odd, we can colour the vertices using four colours such that $\theta(c_1)=\theta(c_2)=\theta(c_3)=\frac{n+1}{2}$ and $\theta(c_4)=\frac{n-1}{2}$ (see the first graph Figure \ref{fig:e-csfl} for illustration). Therefore, $\cn(CSF_n)=\frac{n+1}{2}$ and $\cnc(CSF_n)=\frac{(n+1)^3(n-1)}{16}$.
\end{proof}

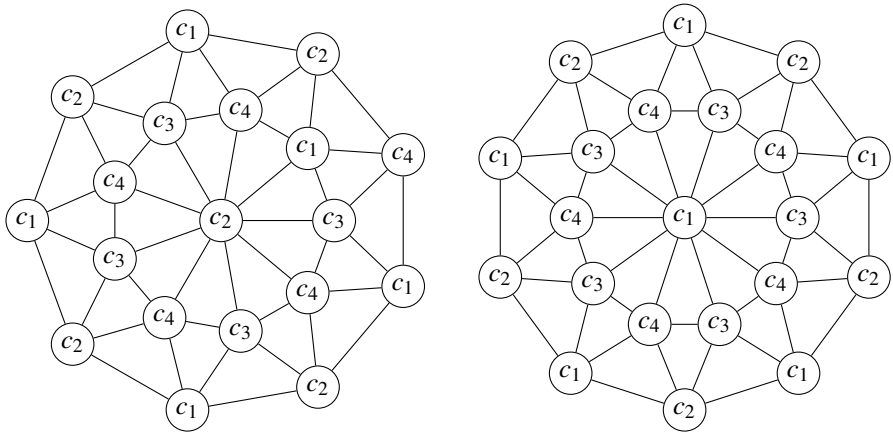
\begin{figure}[h!]
\begin{center}
\begin{tikzpicture}[scale=0.85] 
\vertex (v1) at (0:1.75) []{$c_3$};
\vertex (v2) at (320:1.75) []{$c_4$};
\vertex (v3) at (280:1.75) []{$c_3$};
\vertex (v4) at (240:1.75) []{$c_4$};
\vertex (v5) at (200:1.75) []{$c_3$};
\vertex (v6) at (160:1.75) []{$c_4$};
\vertex (v7) at (120:1.75) []{$c_3$};
\vertex (v8) at (80:1.75) []{$c_4$};
\vertex (v9) at (40:1.75) []{$c_1$};
\vertex (v) at (0:0) []{$c_{2}$};
\vertex (u1) at (20:3) []{$c_4$};
\vertex (u2) at (340:3) []{$c_1$};
\vertex (u3) at (300:3) []{$c_2$};
\vertex (u4) at (260:3) []{$c_1$};
\vertex (u5) at (220:3) []{$c_2$};
\vertex (u6) at (180:3) []{$c_1$};
\vertex (u7) at (140:3) []{$c_2$};
\vertex (u8) at (100:3) []{$c_1$};
\vertex (u9) at (60:3) []{$c_2$};
\path 
(v1) edge (v2)
(v2) edge (v3)
(v3) edge (v4)
(v4) edge (v5)
(v5) edge (v6)
(v6) edge (v7)
(v7) edge (v8)
(v8) edge (v9)
(v9) edge (v1)
(v1) edge (v)
(v2) edge (v)
(v3) edge (v)
(v4) edge (v)
(v5) edge (v)
(v6) edge (v)
(v7) edge (v)
(v8) edge (v)
(v9) edge (v)
(u1) edge (v1)
(u1) edge (v9)
(u2) edge (v1)
(u2) edge (v2)
(u3) edge (v2)
(u3) edge (v3)
(u4) edge (v3)
(u4) edge (v4)
(u5) edge (v4)
(u5) edge (v5)
(u6) edge (v5)
(u6) edge (v6)
(u7) edge (v6)
(u7) edge (v7)
(u8) edge (v7)
(u8) edge (v8)
(u9) edge (v8)
(u9) edge (v9)
(u1) edge (u2)
(u2) edge (u3)
(u3) edge (u4)
(u4) edge (u5)
(u5) edge (u6)
(u6) edge (u7)
(u7) edge (u8)
(u8) edge (u9)
(u9) edge (u1)
;
\end{tikzpicture}
\qquad 
\begin{tikzpicture}[scale=0.85] 
\vertex (v1) at (0:1.75) []{$c_3$};
\vertex (v2) at (324:1.75) []{$c_4$};
\vertex (v3) at (288:1.75) []{$c_3$};
\vertex (v4) at (252:1.75) []{$c_4$};
\vertex (v5) at (216:1.75) []{$c_3$};
\vertex (v6) at (180:1.75) []{$c_4$};
\vertex (v7) at (144:1.75) []{$c_3$};
\vertex (v8) at (108:1.75) []{$c_4$};
\vertex (v9) at (72:1.75) []{$c_3$};
\vertex (v10) at (36:1.75) []{$c_4$};
\vertex (v) at (0:0) []{$c_{1}$};
\vertex (u1) at (18:3) []{$c_1$};
\vertex (u2) at (342:3) []{$c_2$};
\vertex (u3) at (306:3) []{$c_1$};
\vertex (u4) at (270:3) []{$c_2$};
\vertex (u5) at (234:3) []{$c_1$};
\vertex (u6) at (198:3) []{$c_2$};
\vertex (u7) at (162:3) []{$c_1$};
\vertex (u8) at (126:3) []{$c_2$};
\vertex (u9) at (90:3) []{$c_1$};
\vertex (u10) at (54:3) []{$c_2$};
\path 
(v1) edge (v2)
(v2) edge (v3)
(v3) edge (v4)
(v4) edge (v5)
(v5) edge (v6)
(v6) edge (v7)
(v7) edge (v8)
(v8) edge (v9)
(v9) edge (v10)
(v10) edge (v1)
(v1) edge (v)
(v2) edge (v)
(v3) edge (v)
(v4) edge (v)
(v5) edge (v)
(v6) edge (v)
(v7) edge (v)
(v8) edge (v)
(v9) edge (v)
(v10) edge (v)
(u1) edge (u2)
(u2) edge (u3)
(u3) edge (u4)
(u4) edge (u5)
(u5) edge (u6)
(u6) edge (u7)
(u7) edge (u8)
(u8) edge (u9)
(u9) edge (u10)
(u10) edge (u1)
(u1) edge (v1)
(u1) edge (v10)
(u2) edge (v1)
(u2) edge (v2)
(u3) edge (v2)
(u3) edge (v3)
(u4) edge (v3)
(u4) edge (v4)
(u5) edge (v4)
(u5) edge (v5)
(u6) edge (v5)
(u6) edge (v6)
(u7) edge (v6)
(u7) edge (v7)
(u8) edge (v7)
(u8) edge (v8)
(u9) edge (v8)
(u9) edge (v9)
(u10) edge (v9)
(u10) edge (v10)
;
\end{tikzpicture}
\end{center}
\caption{\small Chromatic colouring of a closed sunflower graphs}\label{fig:e-csfl}
\end{figure} 

Let $C_n$ and $C_n'$ be two cycles of order $n$ with vertex sets $\{u_1,u2,\ldots,u_n\}$ and $\{v_1,v_2,\ldots,v_n\}$ respectively . The \textit{antiprism graph}, $A_n$ is the graph $V(A_n)=V(C_n)\cup V(C_n')$ and $E(A_n)=E(C_n)\cup E(C_n')\cup E'$, where $E'$ consists of the edges of the form $v_iu_i, v_iu_{i+1}\in E(A_n)$. The following result determines the  chromatic curling numbers of the antiprism $A_n$.

\begin{theorem}\label{Thm-7}
For the antiprism graph $A_n$, we have 
$$\cn(A_n)=
\begin{cases}
\frac{n}{2}; & \text{if $n$ is even}\\
\frac{n+1}{2}; & \text{if $n$ is odd}.
\end{cases}$$ 
and
$$\cnc(A_n)=
\begin{cases}
\frac{n^4}{16}; & \text{if $n$ is even}\\
\frac{(n^2-1)^2}{16}; & \text{if $n$ is odd}.
\end{cases}$$
\end{theorem}
\begin{proof}~
The chromatic colouring of an antiprism graph consists of $4$ colours, say $c_1,c_2,c_3,c_4$. If $n$ is even, the vertices of the outer cycle of $A_n$ can be given the colours $c_1$ and $c_2$ alternatively and the inner cycle of $CSF_n$ can be given the colours $c_3$ and $c_4$ alternatively. The central vertex can be given the colour $c_1$ (see the second graph Figure \ref{fig:e-apl} for illustration). That is, $\theta(c_1)=\theta(c_2)=\theta(c_3)=\theta(c_4)=\frac{n}{2}$. Therefore, $\cn(CSF_n)=\frac{n}{2}$ and $\cnc(CSF_n)=\frac{n^4}{16}$.

If $n$ is odd, we can colour the vertices using four colours such that $\theta(c_1)=\theta(c_2)=\frac{n+1}{2}$ and $\theta(c_3)=\theta(c_4)=\frac{n-1}{2}$ (see the first graph Figure \ref{fig:e-apl} for illustration). Therefore, $\cn(CSF_n)=\frac{n+1}{2}$ and $\cnc(CSF_n)=\frac{(n^2-1)^2}{16}$.
\end{proof}

\begin{figure}[h!]
\begin{center}
\begin{tikzpicture}[scale=0.85] 
\vertex (v1) at (0:1.75) []{$c_3$};
\vertex (v2) at (320:1.75) []{$c_4$};
\vertex (v3) at (280:1.75) []{$c_3$};
\vertex (v4) at (240:1.75) []{$c_4$};
\vertex (v5) at (200:1.75) []{$c_3$};
\vertex (v6) at (160:1.75) []{$c_4$};
\vertex (v7) at (120:1.75) []{$c_3$};
\vertex (v8) at (80:1.75) []{$c_4$};
\vertex (v9) at (40:1.75) []{$c_1$};
\vertex (u1) at (20:3) []{$c_4$};
\vertex (u2) at (340:3) []{$c_1$};
\vertex (u3) at (300:3) []{$c_2$};
\vertex (u4) at (260:3) []{$c_1$};
\vertex (u5) at (220:3) []{$c_2$};
\vertex (u6) at (180:3) []{$c_1$};
\vertex (u7) at (140:3) []{$c_2$};
\vertex (u8) at (100:3) []{$c_1$};
\vertex (u9) at (60:3) []{$c_2$};
\path 
(v1) edge (v2)
(v2) edge (v3)
(v3) edge (v4)
(v4) edge (v5)
(v5) edge (v6)
(v6) edge (v7)
(v7) edge (v8)
(v8) edge (v9)
(v9) edge (v1)
(u1) edge (v1)
(u1) edge (v9)
(u2) edge (v1)
(u2) edge (v2)
(u3) edge (v2)
(u3) edge (v3)
(u4) edge (v3)
(u4) edge (v4)
(u5) edge (v4)
(u5) edge (v5)
(u6) edge (v5)
(u6) edge (v6)
(u7) edge (v6)
(u7) edge (v7)
(u8) edge (v7)
(u8) edge (v8)
(u9) edge (v8)
(u9) edge (v9)
(u1) edge (u2)
(u2) edge (u3)
(u3) edge (u4)
(u4) edge (u5)
(u5) edge (u6)
(u6) edge (u7)
(u7) edge (u8)
(u8) edge (u9)
(u9) edge (u1)
;
\end{tikzpicture}
\qquad 
\begin{tikzpicture}[scale=0.85] 
\vertex (v1) at (0:1.75) []{$c_3$};
\vertex (v2) at (324:1.75) []{$c_4$};
\vertex (v3) at (288:1.75) []{$c_3$};
\vertex (v4) at (252:1.75) []{$c_4$};
\vertex (v5) at (216:1.75) []{$c_3$};
\vertex (v6) at (180:1.75) []{$c_4$};
\vertex (v7) at (144:1.75) []{$c_3$};
\vertex (v8) at (108:1.75) []{$c_4$};
\vertex (v9) at (72:1.75) []{$c_3$};
\vertex (v10) at (36:1.75) []{$c_4$};
\vertex (u1) at (18:3) []{$c_1$};
\vertex (u2) at (342:3) []{$c_2$};
\vertex (u3) at (306:3) []{$c_1$};
\vertex (u4) at (270:3) []{$c_2$};
\vertex (u5) at (234:3) []{$c_1$};
\vertex (u6) at (198:3) []{$c_2$};
\vertex (u7) at (162:3) []{$c_1$};
\vertex (u8) at (126:3) []{$c_2$};
\vertex (u9) at (90:3) []{$c_1$};
\vertex (u10) at (54:3) []{$c_2$};
\path 
(v1) edge (v2)
(v2) edge (v3)
(v3) edge (v4)
(v4) edge (v5)
(v5) edge (v6)
(v6) edge (v7)
(v7) edge (v8)
(v8) edge (v9)
(v9) edge (v10)
(v10) edge (v1)
(u1) edge (u2)
(u2) edge (u3)
(u3) edge (u4)
(u4) edge (u5)
(u5) edge (u6)
(u6) edge (u7)
(u7) edge (u8)
(u8) edge (u9)
(u9) edge (u10)
(u10) edge (u1)
(u1) edge (v1)
(u1) edge (v10)
(u2) edge (v1)
(u2) edge (v2)
(u3) edge (v2)
(u3) edge (v3)
(u4) edge (v3)
(u4) edge (v4)
(u5) edge (v4)
(u5) edge (v5)
(u6) edge (v5)
(u6) edge (v6)
(u7) edge (v6)
(u7) edge (v7)
(u8) edge (v7)
(u8) edge (v8)
(u9) edge (v8)
(u9) edge (v9)
(u10) edge (v9)
(u10) edge (v10)
;
\end{tikzpicture}
\end{center}
\caption{\small Chromatic colouring of antiprism graphs}\label{fig:e-apl}
\end{figure}
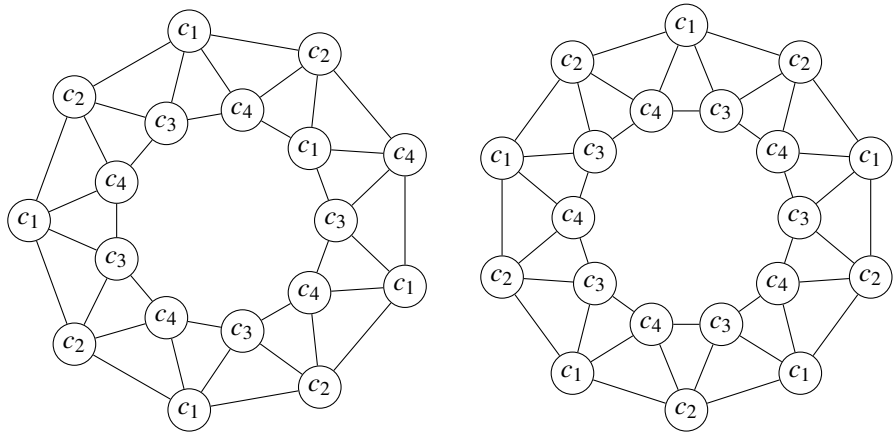 

A \textit{blossom graph} $Bl_n$ is the graph obtained by joining all vertices of the outer cycle of a closed sunflower graph $CSF_n$ to its central vertex.

\begin{theorem}\label{Thm-8}
For the blossom graph $Bl_n$, we have 
$$\cn(Bl_n)=
\begin{cases}
\frac{n}{2}; & \text{if $n$ is even}\\
\frac{n-1}{2}; & \text{if $n$ is odd}.
\end{cases}$$ 
and
$$\cnc(Bl_n)=
\begin{cases}
\frac{n^4}{16}; & \text{if $n$ is even}\\
\frac{(n^2-1)^2}{16}; & \text{if $n$ is odd}.
\end{cases}$$
\end{theorem}
\begin{proof}~
The proof is similar to that of Theorem \ref{Thm-7}. The chromatic colouring of blossoms are illustrated in Figure \ref{fig:e-bll}.
\end{proof}

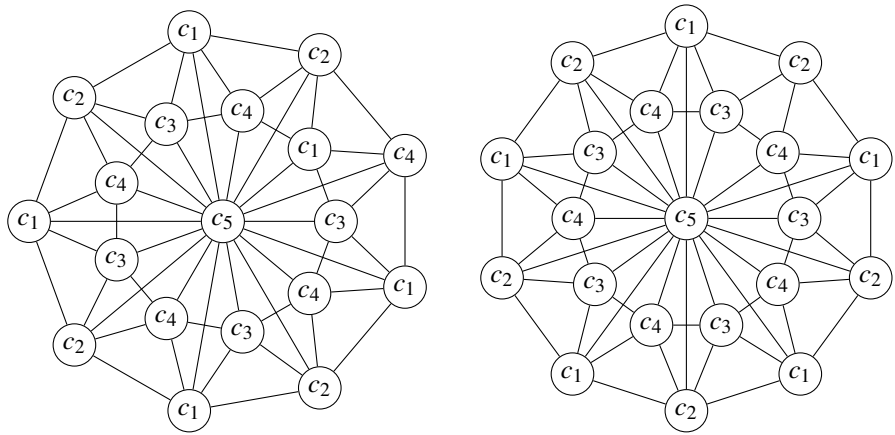
\begin{figure}[h!]
\begin{center}
\begin{tikzpicture}[scale=0.85] 
\vertex (v1) at (0:1.75) []{$c_3$};
\vertex (v2) at (320:1.75) []{$c_4$};
\vertex (v3) at (280:1.75) []{$c_3$};
\vertex (v4) at (240:1.75) []{$c_4$};
\vertex (v5) at (200:1.75) []{$c_3$};
\vertex (v6) at (160:1.75) []{$c_4$};
\vertex (v7) at (120:1.75) []{$c_3$};
\vertex (v8) at (80:1.75) []{$c_4$};
\vertex (v9) at (40:1.75) []{$c_1$};
\vertex (v) at (0:0) []{$c_{5}$};
\vertex (u1) at (20:3) []{$c_4$};
\vertex (u2) at (340:3) []{$c_1$};
\vertex (u3) at (300:3) []{$c_2$};
\vertex (u4) at (260:3) []{$c_1$};
\vertex (u5) at (220:3) []{$c_2$};
\vertex (u6) at (180:3) []{$c_1$};
\vertex (u7) at (140:3) []{$c_2$};
\vertex (u8) at (100:3) []{$c_1$};
\vertex (u9) at (60:3) []{$c_2$};
\path 
(v1) edge (v2)
(v2) edge (v3)
(v3) edge (v4)
(v4) edge (v5)
(v5) edge (v6)
(v6) edge (v7)
(v7) edge (v8)
(v8) edge (v9)
(v9) edge (v1)
(v1) edge (v)
(v2) edge (v)
(v3) edge (v)
(v4) edge (v)
(v5) edge (v)
(v6) edge (v)
(v7) edge (v)
(v8) edge (v)
(v9) edge (v)
(u1) edge (v)
(u2) edge (v)
(u3) edge (v)
(u4) edge (v)
(u5) edge (v)
(u6) edge (v)
(u7) edge (v)
(u8) edge (v)
(u9) edge (v)
(u1) edge (v1)
(u1) edge (v9)
(u2) edge (v1)
(u2) edge (v2)
(u3) edge (v2)
(u3) edge (v3)
(u4) edge (v3)
(u4) edge (v4)
(u5) edge (v4)
(u5) edge (v5)
(u6) edge (v5)
(u6) edge (v6)
(u7) edge (v6)
(u7) edge (v7)
(u8) edge (v7)
(u8) edge (v8)
(u9) edge (v8)
(u9) edge (v9)
(u1) edge (u2)
(u2) edge (u3)
(u3) edge (u4)
(u4) edge (u5)
(u5) edge (u6)
(u6) edge (u7)
(u7) edge (u8)
(u8) edge (u9)
(u9) edge (u1)
;
\end{tikzpicture}
\qquad 
\begin{tikzpicture}[scale=0.85] 
\vertex (v1) at (0:1.75) []{$c_3$};
\vertex (v2) at (324:1.75) []{$c_4$};
\vertex (v3) at (288:1.75) []{$c_3$};
\vertex (v4) at (252:1.75) []{$c_4$};
\vertex (v5) at (216:1.75) []{$c_3$};
\vertex (v6) at (180:1.75) []{$c_4$};
\vertex (v7) at (144:1.75) []{$c_3$};
\vertex (v8) at (108:1.75) []{$c_4$};
\vertex (v9) at (72:1.75) []{$c_3$};
\vertex (v10) at (36:1.75) []{$c_4$};
\vertex (v) at (0:0) []{$c_{5}$};
\vertex (u1) at (18:3) []{$c_1$};
\vertex (u2) at (342:3) []{$c_2$};
\vertex (u3) at (306:3) []{$c_1$};
\vertex (u4) at (270:3) []{$c_2$};
\vertex (u5) at (234:3) []{$c_1$};
\vertex (u6) at (198:3) []{$c_2$};
\vertex (u7) at (162:3) []{$c_1$};
\vertex (u8) at (126:3) []{$c_2$};
\vertex (u9) at (90:3) []{$c_1$};
\vertex (u10) at (54:3) []{$c_2$};
\path 
(v1) edge (v2)
(v2) edge (v3)
(v3) edge (v4)
(v4) edge (v5)
(v5) edge (v6)
(v6) edge (v7)
(v7) edge (v8)
(v8) edge (v9)
(v9) edge (v10)
(v10) edge (v1)
(v1) edge (v)
(v2) edge (v)
(v3) edge (v)
(v4) edge (v)
(v5) edge (v)
(v6) edge (v)
(v7) edge (v)
(v8) edge (v)
(v9) edge (v)
(u1) edge (v)
(u2) edge (v)
(u3) edge (v)
(u4) edge (v)
(u5) edge (v)
(u6) edge (v)
(u7) edge (v)
(u8) edge (v)
(u9) edge (v)
(v10) edge (v)
(u1) edge (u2)
(u2) edge (u3)
(u3) edge (u4)
(u4) edge (u5)
(u5) edge (u6)
(u6) edge (u7)
(u7) edge (u8)
(u8) edge (u9)
(u9) edge (u10)
(u10) edge (u1)
(u1) edge (v1)
(u1) edge (v10)
(u2) edge (v1)
(u2) edge (v2)
(u3) edge (v2)
(u3) edge (v3)
(u4) edge (v3)
(u4) edge (v4)
(u5) edge (v4)
(u5) edge (v5)
(u6) edge (v5)
(u6) edge (v6)
(u7) edge (v6)
(u7) edge (v7)
(u8) edge (v7)
(u8) edge (v8)
(u9) edge (v8)
(u9) edge (v9)
(u10) edge (v9)
(u10) edge (v10)
;
\end{tikzpicture}
\end{center}
\caption{\small Chromatic colouring of a closed sunflower graphs}\label{fig:e-bll}
\end{figure}

\section{Conclusion}

In this paper, we have introduced a new parameter, namely, chromatic curling number of graphs as a colouring analogue of the curling number of graphs.  We have also determined the chromatic curling number of paths, cycles and certain cycle related graphs.  The problems of determining this parameter for certain other graph classes, derived graph classes, graph operations, graph products and graph powers are open.  The notion of chromatic curling number of graphs can be extended to other types of graph colourings also.  All these facts highlights a wide scope for further research in this area. 


}
\end{document}